\documentclass[10pt]{article}

\usepackage{amssymb,amsmath,amsopn,amsthm}
\usepackage[mathscr]{euscript}
\usepackage[margin=1.5in]{geometry}
\usepackage{graphicx}
\usepackage{float}
\usepackage{tikz-cd}
\usepackage[utf8]{inputenc}
\usepackage{hyperref}

\newtheorem{lemma}{Lemma}
\theoremstyle{definition}
\newtheorem{defn}{Definition}
\newtheorem{Theorem}{Theorem}

\theoremstyle{remark}
\newtheorem{rmk}{Remark}
\newtheorem{example}{Example}

\title{Topological perspective on Statistical Quantities I}
\date{}
\author{Nissim Ranade}

\begin{document}
\maketitle

\section*{Introduction}
In statistics cumulants are defined to be functions that measure dependence of random variables. If the random variables are independent the cumulants are zero. The idea of cumulants can be generalized to non-commutative probability theory by defining the Boolean cumulants. For a map $e$ between associative algebras these maps measure the deviation of $e$ from being an algebra map. 

The Boolean cumulants are a family of maps $K_n:V^{\otimes n}\rightarrow \mathbb{K}$ defined as follows. 
\[K_1(a) = e(a)\]
\[K_2(a,b) =  e(ab)-e(a)e(b) \]
\[K_3(a,b,c) = e(abc)-e(a)e(bc) -e(ab)e(c)+e(a)e(b)e(c) \]

$K_n$ in general is given by the following formula. 
\[K_n(a_1, a_2,\ldots,a_n) = \sum \pm e(a_1,\ldots,a_i)e(a_{i+1},\ldots)\ldots e(\ldots,a_n) \] 
The above sum is taken over all ordered partitions of $n$. The even partitions occur with negative signs and the odd partitions occur with positive signs. Expectations of products can also be computed using Boolean cumulants. If $e$ is a map of algebras then the cumulants are all zero.

More generally the Boolean cumulants can be defined using the above formulas for chain maps between differential graded algebras. For instance, consider the differential forms $\Omega(M)$ on a manifold $M$ and the cochains $C^*(M)$ on a discrete simplicial structure on the manifold. There is a chain map $I$ from $\Omega(M)$ to $C^*(M)$ given by integrating forms on the simplices. This map induces an isomorphism on the cohomologies of the two complexes. The differential forms have an algebra structure given by the wedge product. An associative cup product can be defined on $C^*(M)$ but $I$ is not  a map of algebras for this product. Both the products however induce products on cohomology and the isomorphism induced by $I$ on cohomology respects the induced products. Thus while the cumulants exist at the level of cochain complexes they vanish on cohomology. 

While $I$ is not a map of algebras it happens to be the first term of what is called an $A_\infty$ morphism. Associative algebras are a special case of $A_\infty$ algebras which are algebras that are associative up to infinite homotopy. $A_\infty$ morphisms are morphisms of these algebras. We have the following theorem which relate the Boolean cumulants to the structure of $A_\infty$ morphisms.

\begin{Theorem}
Let $(A,\wedge_A, d_A)$ and $(B,\wedge_B, d_B)$ be two \textit{dga}s. Let $p$ be a chain map from $A$ to $B$. Let $K_2$, $K_3$ and so on be the Boolean cumulants of $p$. Suppose $p$ is the first term of an $A_\infty$ morphism $(p, p_2, p_3, \ldots)$ where $p_n:A^{\otimes n}\rightarrow B$. Then the following statements hold.
\begin{itemize}
\item[i)] $p_2$ gives a homotopy from the second Boolean cumulant $K_2$ to zero. All the higher Boolean cumulants $K_n$ are also homotopic to zero using maps created by $p_2$ and $p_1$.
\item[ii)] $p_3$ gives a homotopy between different ways of making $K_3$ homotopic to zero. For all the higher Boolean cumulants, homotopies between the multiple different ways of making them homotopic to zero are homotopic to each other using $p_3$, $p_2$ and $p_1$.
\item[iii)] In general any cycles that are created using the homotopies $\{p_j\}_{j=1}^n$ are made homotopic to zero using maps made by $\{p_j\}_{j=1}^{n+1}$.
\end{itemize}

\end{Theorem}

The above theorem means that if $p$ is the first term of an $A_\infty$ morphism then the cumulants of $p$ completely collapse. That is, they are not only homotopic to zero, multiple homotopies are homotopic to each other. The definition of Boolean cumulants can be generalized to $A_\infty$ algebras in general. In this case while they are only defined up to homotopy. For the Boolean cumulants of a map between $A_\infty$ algebras we have the following theorem.

\begin{Theorem}
Let $A$ and $B$ be two $A_\infty$ algebras. Let $p$ be a chain map from $A$ to $B$. Let $K_2$, $K_3$ and so on be the Boolean cumulants of $p$ defined up to homotopy. Suppose $p$ is the first term of an $A_\infty$ morphism $(p, p_2, p_3, \ldots)$ where $p_n:A^{\otimes n}\rightarrow B$. Then the following statements hold.
\begin{itemize}
\item[i)] $p_2$ gives a homotopy from the second Boolean cumulant $K_2$ to zero. All the different ways of defining the higher Boolean cumulants $K_n$ are also homotopic to zero using maps created by $p_2$ and $p_1$.
\item[ii)] $p_3$ gives a homotopy between different ways of making $K_3$ homotopic to zero. For all the higher Boolean cumulants, homotopies between the multiple different ways of making them homotopic to zero are homotopic to each other using $p_3$, $p_2$ and $p_1$.
\item[iii)] In general any cycles that are created using the homotopies $\{p_j\}_{j=1}^n$ are made homotopic to zero using maps made by $\{p_j\}_{j=1}^{n+1}$.
\end{itemize}

\end{Theorem}

\section{$A_\infty$ algebras and their morphisms}

In 1963 James Stasheff defined a notion of an algebra that was associative up to 'infinite homotopy'. 

\begin{defn}
An $A_\infty$ algebra is a graded vector space $A$ with a collection of linear maps 
\[m_n:A[1]^{\otimes n} \rightarrow A[1]\]
such that $m_n$ have degree $1$ on and they satisfy the following equations for every $n$
\begin{equation}
\label{A-infinity equations}
\sum_{i+j=n} m_i(1\otimes 1\otimes \ldots m_j\ldots \otimes 1) = 0
\end{equation}
\end{defn}

The equations \ref{A-infinity equations} imply the following statements.
\begin{itemize}
\item $m_1$ is a linear map of degree $1$ that squares to zero. Thus $m_1$ is a differential on $A$.  
\item $m_2$ is a binary product and $m_1$ is a derivation of this binary product. 
\item Since $m_2$ is not associative, that associator $m_2(m_2\otimes 1)-m_2(1\otimes m_2)$ is not zero. $m_3$ is a map whose boundary is the associator. That is $m_3$ makes $m_2$ homotopic to being associative.
\item $m_n$, for $n$ larger than $3$, makes cycles created by $m_k$, for $k$ less than $n$, homotopic to zero. 
\end{itemize}

The homotopies given by $m_n$ can be described using polyhedrons described by Stasheff. For instance $m_3$ is a homotopy between the two terms of the associator and is described by a line. There are five different ways of combining four terms using a binary product and they correspond to five vertices of a pentagon that is used to describe $m_4$. The first three associahedra are described as follows. 

\begin{figure}[H]
\centering \includegraphics[width=5in]{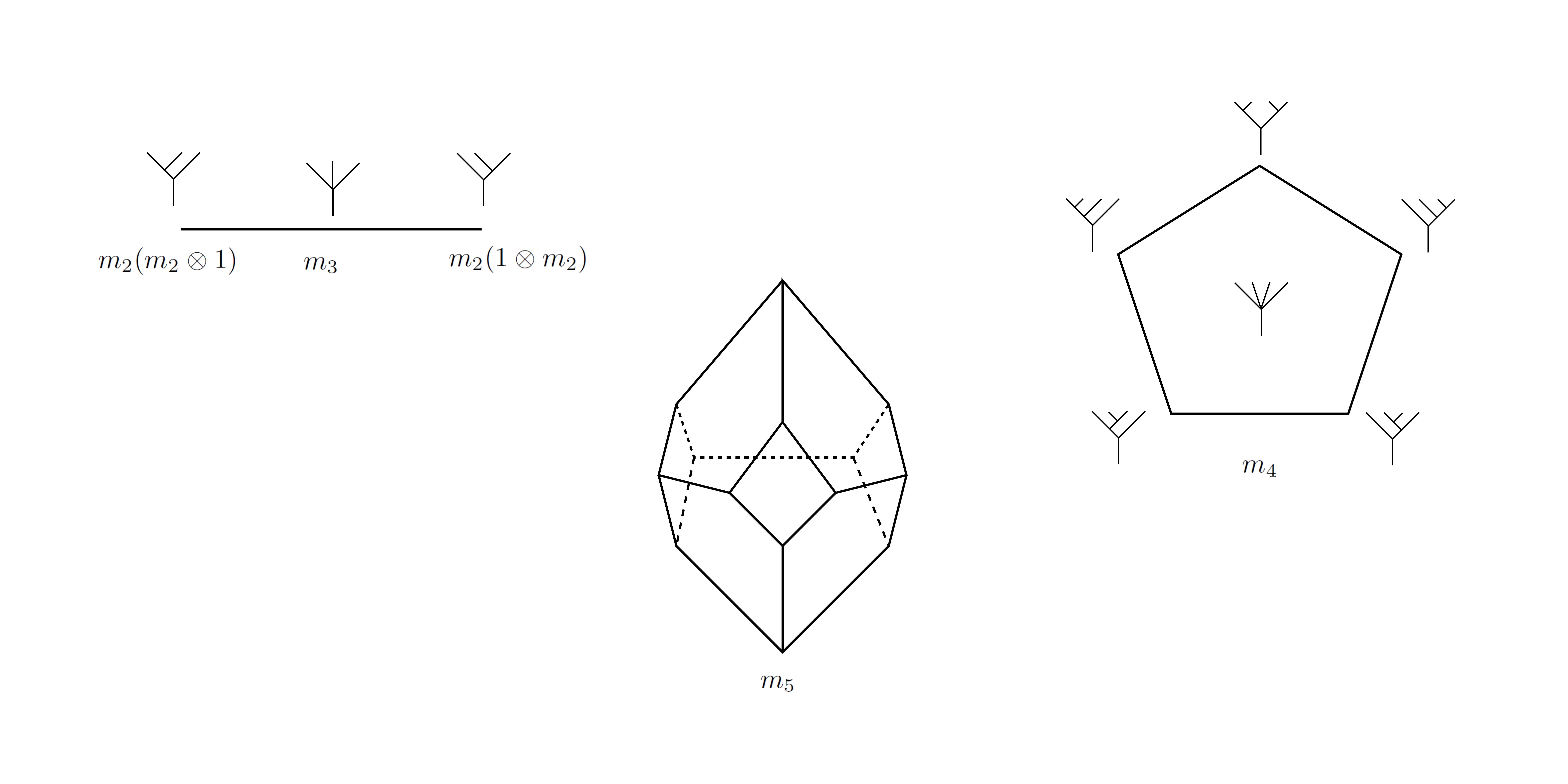}
\caption{}
\end{figure}

\begin{defn}
An \textit{$A_\infty$ morphism} $P$ between $A_\infty$ algebras $(A, m_1^A, m_2^A, \ldots)$ and $(B, m_1^B, m_2^B, \ldots)$ is a collection of linear maps 
\[p_n:A^{\otimes n} \rightarrow B\]
such that
\[\sum_{k=1}^n \sum_{n_1+\ldots+n_k=n} m_k^B(p_{n_1}\otimes\ldots \otimes p_{n_k})\]\[ = \sum_{k=1}^n \sum_{j=0}^{n-k} p_{n-k+1}(1\otimes \ldots m_k^A \ldots \otimes 1) \]

\end{defn}

\section{Structure of an $A_\infty$ morphism between \textit{dga}s}

Recall that between associative algebras without differentials, every $A_\infty$ morphism is in fact an algebra morphism. This is however not necessarily the case when we consider $A_\infty$ morphisms between \textit{dga}s. 

Suppose $(A,d_A,\wedge_A)$ and $(B,d_B,\wedge_B)$ are two differential graded algebra. Recall that by definition an $A_\infty$ morphism is a collection of maps$(p_1,p_2,\ldots)$, $p_n:A^{\otimes n}\rightarrow B$ where  which satisfy the following compatibility relations for every $n$. 
\[\sum_{i+j=n}\wedge_B(p_i\otimes p_j)+d_B\circ p_n = \sum p_{n-1}(1\otimes \ldots \wedge_A \ldots 1)+p_n(1\otimes \ldots d_A\ldots\otimes 1)  \] In particular for $n=1$ the compatibility relation is as follows.
\begin{equation}
\label{A_infty compatibility}
\wedge_B(p_1\otimes p_1) + d_B\circ p_2 = p_1\circ\wedge_A + p_2(d_A\otimes 1+1\otimes d_A)
\end{equation} 
Also recall that for a map $p_n:A^{\otimes n}\rightarrow B$, the differential of $p_n$ in the space $Hom(A^{\otimes n}, B)$ and is defines as 
\begin{equation}
\label{boundry}
(p_n)= d_B\circ p_n + (-1)^{n+1}p_n(1\otimes\ldots d_A \ldots \otimes 1)
\end{equation} 

We call this the boundary of the map $p_n$. Note that since $\partial(p_1) = 0$ which implies $p_1$ is a chain map.

\begin{lemma}
The Boolean cumulants $K_2$, $K_3$ and so on of the map $p_1:A\rightarrow B$ are boundaries of maps that can be constructed using the map $p_2$.
\end{lemma}

\begin{proof}
We will prove this lemma by induction. For $a$ and $b$ in $A$, 
\[K_2(a,b) = p_1(\wedge_A(a,b) - \wedge_B(p_1(a), p_1(b)))\] For simplicity of notation we will suppress $\wedge_A$ and $\wedge_B$. Thus the formula for the cumulants is now more familiar.\[K_2(a,b) = p_1(ab) - p_1(a)p_1(b)\] Thus from equations \ref{A_infty compatibility} and \ref{boundry} we have that \[\partial(p_2)(a,b) = K_2(a,b)\] In general we know that 
\[K_n(a_1, a_2,\ldots,a_n) = \sum_{\text{ordered partitions of }n} \pm p_1(a_1\ldots a_i)p_1(a_{i+1}\ldots)\ldots p_1(\ldots a_n) \] In general we can describe $K_n$ in terms of $K_{n-1}$ and $p_1$ as follows. \[K_n(a_1, a_2, \ldots, a_n) = K_{n-1}(a_1a_2, a_3, \ldots, a_n) - p_1(a)K_{n-1}(a_2,\ldots,a_n)\] Since $K_{n-1}$ can be written as a boundary of some map $f$  and $\partial(p_1)=0$ we have \[K_n = \partial(f\circ (\wedge_A\otimes id)) - p_1\otimes \partial(f) = \partial(f\circ (\wedge_A\otimes id) - p_1\otimes f) \] This proves that all the cumulants are boundaries in the Hom-complex. 

\end{proof}

Note that for $K_3$, $K_4$ and so on there is not a unique way to write $K_n$ as a boundary of a map. Given that $K_2$ is the boundary of $p_2$, $K_3$ can be describes as the boundary of two different maps. \[K_3(a,b,c)= \partial(p_2(a b,c)-p_1(a) p_2(b,c))\] \[=\partial(p_2(a, b c) - p_2(a,b) p_1(c)) \] Similarly $K_4$ can be described as a boundary of multiple different maps. 

The terms of the $n$th cumulant correspond the the ordered partitions of $n$. We associate a graph $G_n$ to $K_n$. The vertices of $G_n$ correspond to terms of $K_n$ (or equivalently to ordered partitions of $n$). Two vertices are connected to each other via an edge for the corresponding partitions, one partition can be obtained from the other by splitting one of the sub strings. Note that the vertices of $G_n$ correspond to all the different ways of combining $n-1$ ordered inputs from $A$ using $p$ and the binary products to give exactly one output in $B$. If $p$ were an algebra map all of these ways would be equal.  

\begin{figure}[H]
\centering \includegraphics[width=5in]{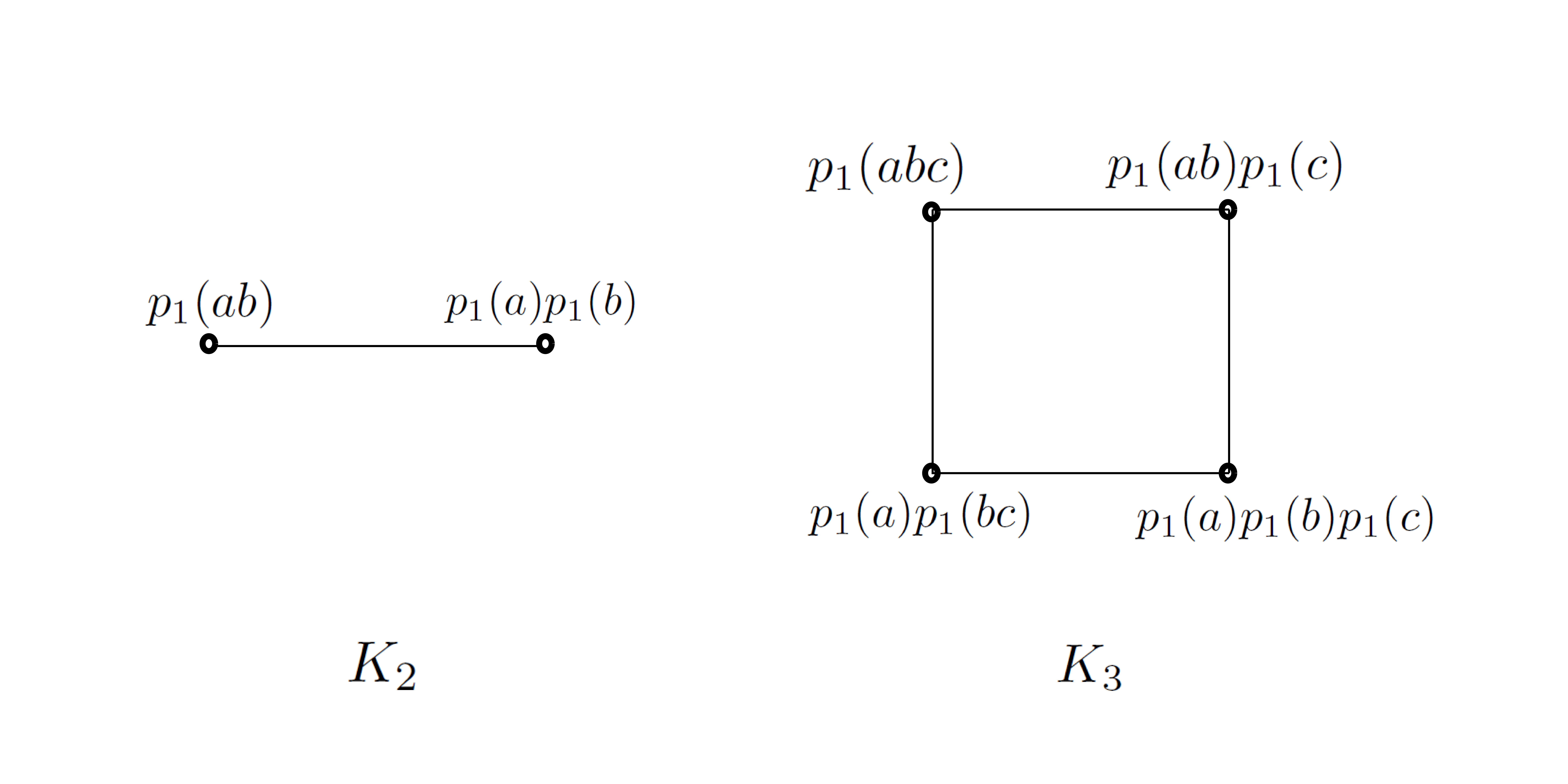}
\caption{}
\end{figure}

\begin{lemma}
The graph $G_n$ is the one skeleton of an $n-1$-cube. 
\end{lemma}

\begin{proof}
We will prove this by induction. Note that $K_3$ is a square and recall \[K_n(a_1, a_2, \ldots, a_n) = K_{n-1}(a_1a_2, a_3, \ldots, a_n) - p_1(a_1)K_{n-1}(a_2,\ldots,a_n)\] By induction hypothesis the subgraphs of $G_n$ corresponding to the above two terms are a $n-2$-cubes (as $G_{n-1}$ is an $n-2$ cube. Edges that go between these subgraphs correspond to splitting sub-strings of the form $a_1a_2\ldots a_i$ into $a_1$ and $a_2\ldots a_i$. Thus for these edges give a one to one correspondence between the vertices of the two $n-2$ cubes. It is easy to check that the adjacent vertices in the first cube go to adjacent vertices in the second cube. Thus the graph of $G_n$ is an $n-1$ cube. 

\end{proof}

If two vertices of $G_n$ are connected by an edge then they occur with opposite signs in $K_n$. Also, the corresponding terms of the cumulant are the boundary of a map involving $p_1$ and $p_2$. For instance $ p_1(ab)p_1(c)-p_1(a)p_1(b)p_1(c)$ is the boundary of the map $p_2(a,b)p_1(c)$ and $p_1(abc)-p_1(a)p_1(bc)$
is the boundary of $p_2(a,bc)$. This is true because the differentials are derivations of the binary product and $p_1$ is a chain map. Thus we can label the edges of $G_n$ with the corresponding maps involving $p_2$. Thus cycles in $G_n$ correspond to cycles in $Hom(A^{\otimes n}, B)$. For instance the following map is the sum of the maps corresponding to the four edges of $G_3$. \[p_2(ab,c)-p_1(a)p_2(b,c) - p_2(a,bc) + p_2(a,b)p_1(c)\] This map is a cycle. 

Note that this map is essentially all the ways of composing the maps $p_2$ and $p_1$ withe the binary products. From the compatibility relation for $p_3$ we get that \[\partial(p_3)(a,b,c) = p_2(ab,c)-p_1(a)p_2(b,c) - p_2(a,bc) + p_2(a,b)p_1(c)\]

\begin{lemma}
The cycle corresponding to the squares in the cubes $G_n$ are boundaries of maps constructed using $p_3$, $p_2$ and $p_1$. 
\end{lemma}
 
\begin{proof}
In general a square in $G_n$ is made with four vertices which differ in partitions added at two positions. There are two cases to consider. First is when a single substring is split into three in two different ways. Both these cases and the maps that give the homotopies to zero are shown in the following diagrams

\begin{figure}[H]
\centering \includegraphics[width=5in]{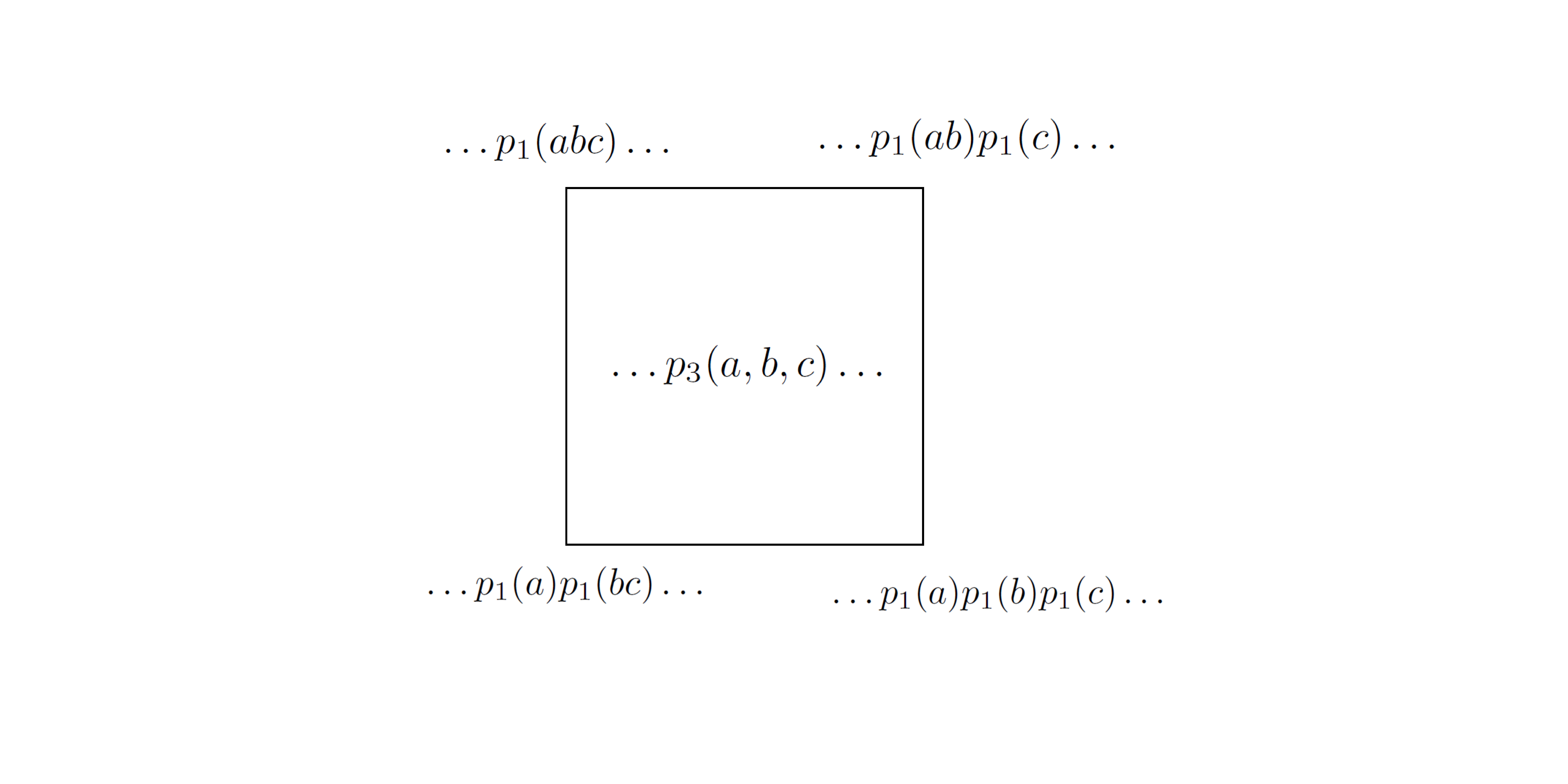}
\caption{\label{p3}}
\end{figure}

\begin{figure}[H]
\centering \includegraphics[width=5in]{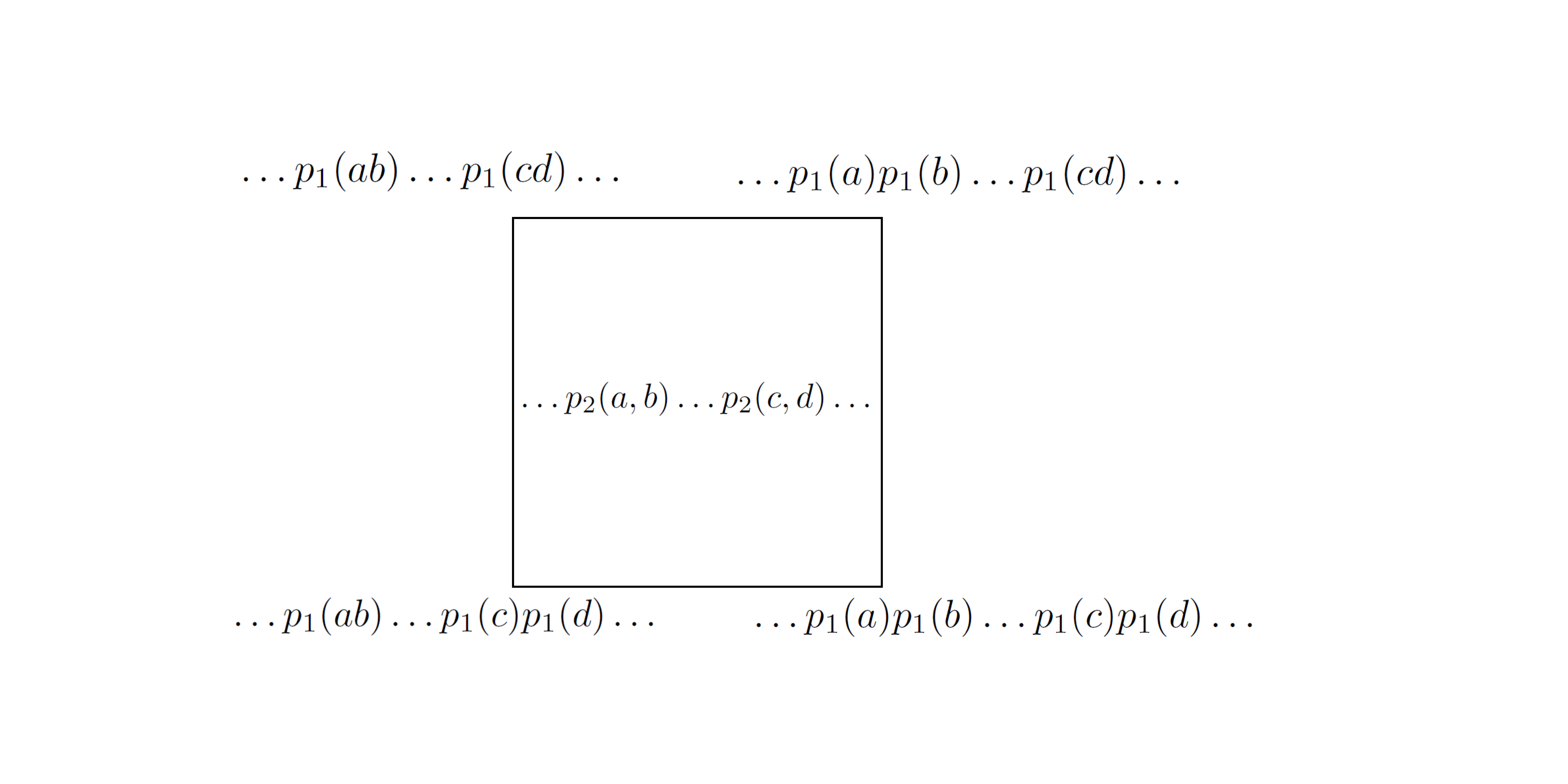}
\caption{\label{p2,p2}}
\end{figure}

\end{proof}
Let $g_n$ be an $n-2$-dimensional solid cube such that $G_n$ is its one skeleton. Then from the above lemma we can associate to the $2$-cells of $g_n$ maps made using $p_3$ and $p_2$. We are now ready to state and prove our theorem in the context of associative algebras.

\begin{Theorem}
Let $(A,\wedge_A, d_A)$ and $(B,\wedge_B, d_B)$ be two \textit{dga}s. Let $p$ be a chain map from $A$ to $B$. Let $K_2$, $K_3$ and so on be the Boolean cumulants of $p$. Suppose $p$ is the first term of an $A_\infty$ morphism $(p, p_2, p_3, \ldots)$ where $p_n:A^{\otimes n}\rightarrow B$. Then the following statements hold.
\begin{itemize}
\item[i)] $p_2$ gives a homotopy from the second Boolean cumulant $K_2$ to zero. All the higher Boolean cumulants $K_n$ are also homotopic to zero using maps created by $p_2$ and $p_1$.
\item[ii)] $p_3$ gives a homotopy between different ways of making $K_3$ homotopic to zero. For all the higher Boolean cumulants, homotopies between the multiple different ways of making them homotopic to zero are homotopic to each other using $p_3$, $p_2$ and $p_1$.
\item[iii)] In general any cycles that are created using the homotopies $\{p_j\}_{j=1}^n$ are made homotopic to zero using maps made by $\{p_j\}_{j=1}^{n+1}$.
\end{itemize}

\end{Theorem}

\begin{proof}
The previously proved lemmas prove the first two parts of this theorem. In general $2$ cycles created by $p_2$ and $p_3$ correspond to $2$ cycles in $g_n$. Consider the boundary of $p_n$ in general. Recall that from by definition $p_n$ satisfies the equation. 
\[\sum_{k=1}^n \sum_{n_1+\ldots+n_k=n} m_k^B(p_{n_1}\otimes\ldots \otimes p_{n_k})\]\[ = \sum_{k=1}^n \sum_{j=0}^{n-k} p_{n-k+1}(1\otimes \ldots m_k^A \ldots \otimes 1) \] Since in this case $m_k$ are all zero except for $k=1$ and $k=2$ we get
\[d(p_n)+\sum_{n_1+n_2=n} \wedge_B(p_{n_1}\otimes p_{n_2}) \]\[= \sum_{k=1}^np_n(1\otimes\ldots d\ldots\otimes 1) + \sum_{k=1}^{n-1}p_{n-1}(1\otimes\ldots \wedge_A\ldots\otimes 1 )\] By rearranging the terms of the above equation we find $\partial(p_n)$.
\[d(p_n)-\sum_{k=1}^np_n(1\otimes\ldots d\ldots\otimes 1)\]\[=  \sum_{k=1}^{n-1}p_{n-1}(1\otimes\ldots \wedge_A\ldots\otimes 1)-\sum_{n_1+n_2=n} \wedge_B(p_{n_1}\otimes p_{n_2})\]

Also  

\[\partial(\wedge_B(p_{n_1}\otimes p_{n_2})) = \wedge_B(\partial(p_{n_1})\otimes \partial(p_{n_2}))\]

Thus in general to a map of the type $p_{j_1}p_{j_2}\ldots p_{j_m}$ we associate a cell of dimension $j_1+j_2\ldots j_m - m$ which is attached in $g_n$ to the cycle corresponding to its boundary.

\begin{figure}[H]
\centering \includegraphics[width=5in]{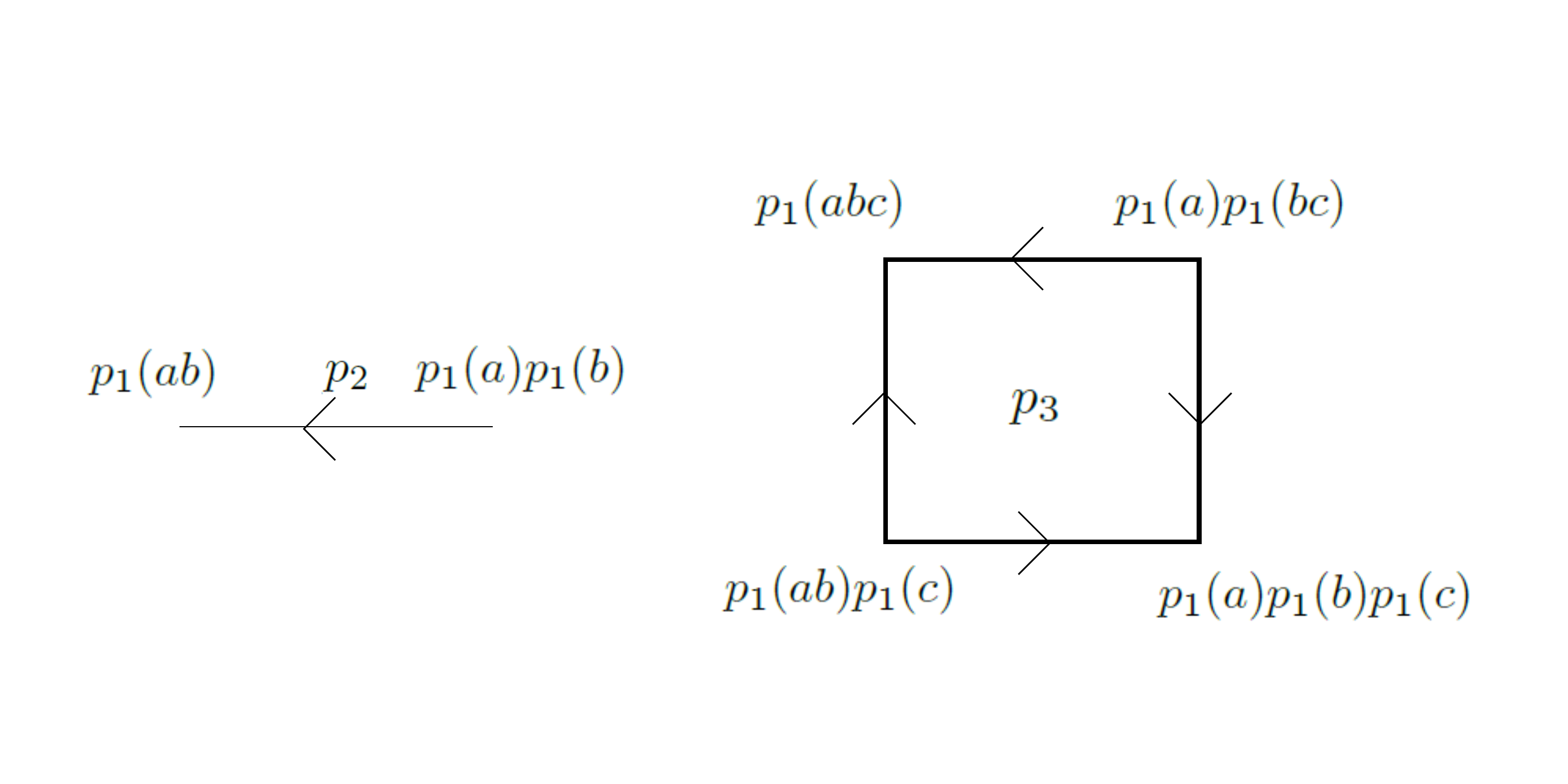}
\caption{$p_2$ and $p_3$}
\end{figure}

\begin{figure}[H]
\centering \includegraphics[width=5in]{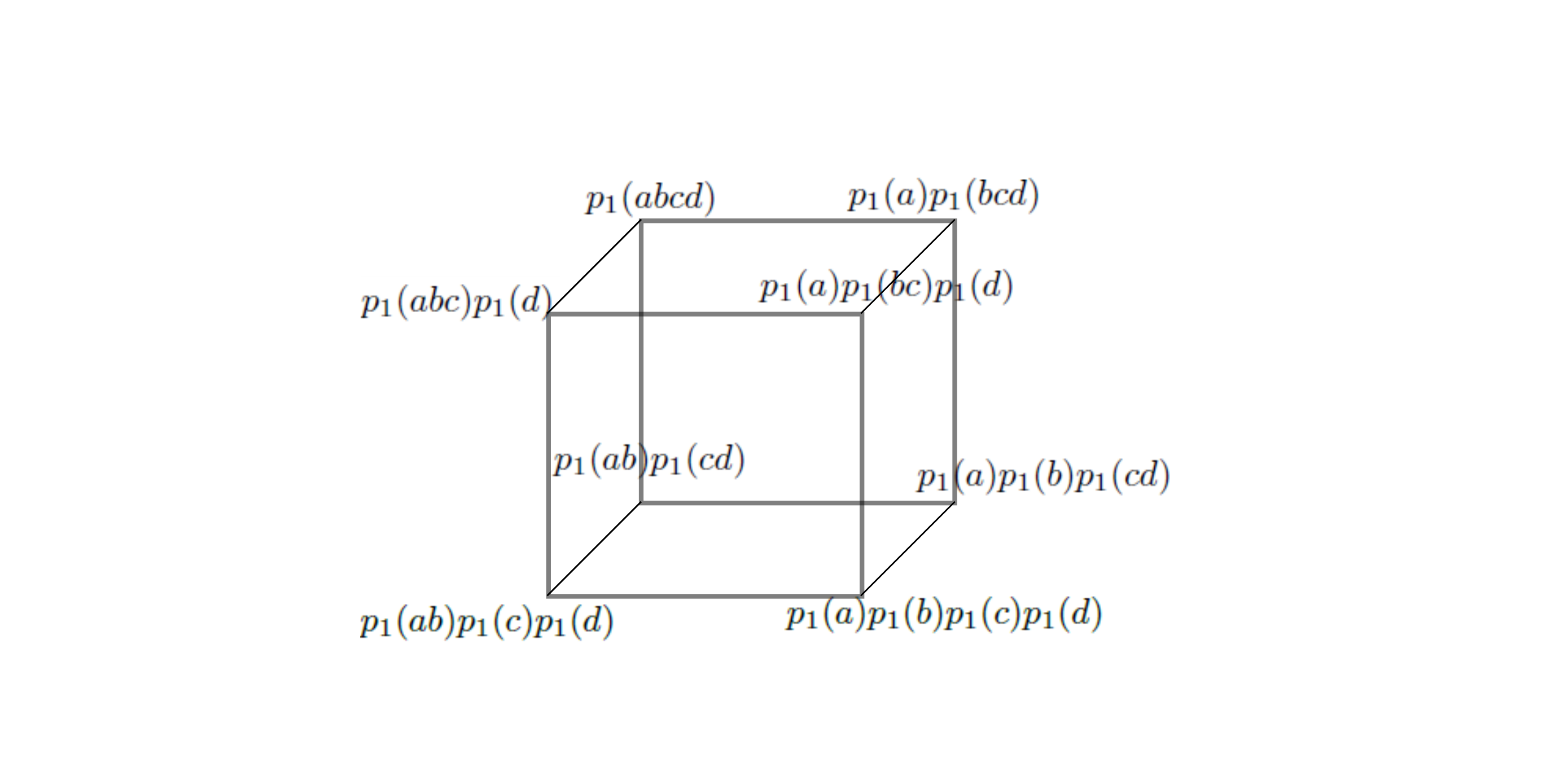}
\caption{$p_4$}
\end{figure}

Since $g_n$ are solid cubes, they are contractible. Also, all the cells of $g_n$ correspond to either a function of the form $p_1\ldots p_k\ldots p_1$ or a function of the form $p(a_1)\ldots p_k\ldots p_l \ldots p(a_n)$. Thus we have that all cycles created by $\{p_k\}$ are contractible.  

\end{proof}

\section{Differential forms and cochains on a manifold $M$}

\begin{defn} 
A \textit{differential graded associative algebra} or a \textit{dga} is a triple $(A,d,m)$ such that 
\begin{itemize}
\item[i)] $A = \bigoplus_{n\in \mathbb{Z}} A_n$ is a graded vector space. 
\item[ii)] $m:A\otimes A\rightarrow A$ is an associative product of degree zero. That is $m$ is associative and for $a\in A_n$ and $b\in A_m$, $m(a\otimes b)$ is in $A_{n+m}$.
\item[iii)] $d:A \rightarrow A$ is a linear map of degree $1$ (for $a \in A_n$, $d(a)\in A_{n+1}$) such that $d^2 = 0$. 
\item[iv)] (Leibniz Rule) $d$ and $m$ satisfy the following compatibility relationship \[d(m(a\otimes b)) = m(d(a)\otimes b) + (-1)^{|a|} m(a\otimes d(b)) \]
\end{itemize}

\end{defn}

Note that a \textit{dga} is an $A_\infty$ algebra in particular. The map $m_1$ is the differential of the \textit{dga}. The associative product is the map $m_2$. The maps $m_3$, $m_4$ and so on are all zero. A morphism of \textit{dga}s is an $A_\infty$ morphism in particular. However, an $A_\infty$ morphism between two \textit{dga}s is not necessarily a algebra morphism.

\begin{rmk}[Koszul sign convention]
Given two linear maps $f$ and $g$ of graded vector space we can define the $f\otimes g$ to be a map from the tensor product of the domain to the tensor product of the range. We use the Koszul sign convention when applying tensor products of linear maps. That is \[f\otimes g(x\otimes y) = (-1)^{|x||g|}(f(x)\otimes g(y))\]
\end{rmk}

\begin{example}
One of the first non-trivial example of a \textit{dga} is the algebra of differential forms on $\Omega^*(M)$ on a manifold $M$. The differential $d$ has degree one as $d$ of an $n$ form is an $n+1$ form. The product $m$ is the wedge product which has degree zero as the wedge product of an $m$-form with an $n$-form is a $(m+n)$-form. This algebra is also graded commutative as give two forms $\omega$ and $\eta$, \[\omega\wedge\eta = (-1)^{|\omega||\eta|}\eta\wedge\omega \]The wedge product induces the graded commutative cup product on the cohomology of the manifold.
\end{example}

\begin{defn}
A \textit{differential graded coalgebra} or a \textit{dg-coalgebra} is a triple $(C,\delta, \Delta)$ where
\begin{itemize}
\item[i)]  $C = \bigoplus_{n\in \mathbb{Z}} C_n$ is a graded vector space.
\item[ii)] $\Delta:C\rightarrow C\otimes C$ is a co-associative coproduct of degree zero. Coassociativity implies that \[(\Delta\otimes 1)\circ\Delta = (1\otimes \Delta)\circ\Delta \]
\item[iii)] $\delta:C \rightarrow C$ is a linear map of degree $-1$ (for $a \in C_n$, $\delta(a)\in C_{n-1}$) such that $\delta^2 = 0$. 
\item[iv)] (Leibniz Rule) $\delta$ and $\Delta$ satisfy the following compatibility relationship 
\[(\delta \otimes 1 + 1\otimes \delta)\circ \Delta = \Delta \circ \delta \]
\end{itemize}
\end{defn}

\begin{example}
\label{alexander-whitney}
Consider the chains $C_*$ on a finite simplicial decomposition of a space $X$. There is a coproduct map $\Delta:C_*\rightarrow C_*\otimes C_*$ called the Alexander-Whitney map which is given by the following formula on a simplex $ [v_0 ,v_1,\ldots, v_n] $.
\[\Delta([v_0,v_1,\ldots, v_n]) =  \sum_i [v_0,v_1,\ldots,v_i]\otimes [v_i,\ldots,v_n]\]
This product dualizes to an associative product $\mu$ on the cochains $C^*$. The associativity follows from the fact that the coproduct $\Delta$ is coassociative. Also the coproduct satisfies the co-Leibniz property with respect to the boundary operator $\partial$. That is, for a simplex $\sigma$ 
\[\Delta(\partial(\sigma)) = (1\otimes\partial + \partial\otimes 1)(\Delta(\sigma))\]
This implies that the co-boundary map $\delta$ is a derivation of the dual product $\mu$ on the cochains $C^*$. Thus $(C^*,\delta,\mu)$ is a differential graded algebra. Unlike the differential forms the cochains are not graded commutative, but the product $\mu$ also induces a graded commutative product on the cohomology of the space. 

\end{example} 

 \begin{example}[Map from differential forms to cochains]
Consider the algebra of differential forms $\Omega^*$ on a manifold $M$ and the cochains $C^*$ on a finite simplicial decomposition of $M$. The differential forms are a $dgca$ and the cochains have the Alexander-Whitney cup product which makes them a $dga$. Consider the map $I:\Omega^* \rightarrow C^*$ defined as follows. For a differential for $\omega$ and a simplex $\sigma$ \[I(\omega)(\sigma)= \int_{\sigma}\omega\] From Stoke's theorem it follows that $I$ is a chain map. It does not respect the product structure at the level of complexes. However, by the de Rahm's theorem $I$ induces an isomorphism on cohomology. The induced isomorphism is in fact a map of the algebra structures. Thus the Boolean cumulants of $I$ are defined on chains, but they vanish on cohomology since the induced map is an algebra map. 
\end{example}

In 1978 V. K. A. M. Gugenheim constructed an $A_\infty$ morphism whose first Taylor coefficient is $I$ \cite{gugenheim1978}. This construction uses iterated integrals as defined by Kuo-Tsai Chen \cite{chen1977}. We will consider the special case of forms and cochains on the interval $[0,1]$. The details of the case are worked out in the paper by Ruggero Bandiera and Florian Schaetz \cite{formsoninterval}

The $0$ cochains on $[0,1]$ are functions on the set $\{0,1\}$ and $1$ cochains are given by one generator corresponding to the one cell. We will call this generator $dt$. Thus a $1$ cochain is of the form $rdt$ where $r$ is in $\mathbb{R}$. The map $p$ is given for a zero form by taking the restriction of the function to the points $0$ and $1$. On the one forms it is given as follows.

\[I(f(x)dx) = (\int_0^1f(x)dx)dt\]

Recall that the associative cup product on the cochains is defined as follows. For two zero forms the cup product is the product of the two functions. For a zero form $F$ and a one form $rdt$ we have \[ F\cup rdt = F(0)rdt\]\[rdt\cup F = 0\] and the cup product of two one forms is zero. Note that this product is not associative and the map $I$ is not a map of algebras. We define the map $I_n:\Omega([0,1])^{\otimes n}\rightarrow C^*([0,1]$ as follows. If any of the inputs of $I_n$ is a zero form then $I_n$ is zero. For $n$ one forms \[I_n(f_1(x)dx,f_2(x)dx\ldots, f_n(x)dx) \]\[= ( \int_{t_1\leq t_2\leq \ldots \leq t_n} f_1(t_1)f_2(t_2)\ldots f_n(t_n)dt_1dt_2\ldots dt_n) dt\] $(I,I_2,I_3,\ldots)$ is an $A_\infty$ morphism from the differential forms to the cochains. 

Note that all the one forms on the interval are exact. Suppose $df_1$ and $df_2$ are exact forms then \[I_2(df_1,df_2) = I_2(d(f_1,df_2)) = \partial(I_2)(f_1,df_2) = K_2(f_1,df_2)\] In general for forms $df_1$, $df_2$, $df_3$ and so on we have \[I_n(df_1,df_2,\ldots,df_n) = I_n(d(f_1,df_2,\ldots,df_n)) = \partial(I_n)(f_1,df_2,\ldots, df_n)\] The above expression is equal to \[I_1(f_1)I_{n-1}(df_2,\ldots, df_n) \pm I_{n-1}(f_1df_2,\ldots, df_n) \] We can compute this expression by induction on $n$. 

For a general simplex $\Delta^n$of dimension $n$, and the map $I:\Omega(\Delta^n)\rightarrow C^*(\Delta^n)$ maps $I_n$ can by defined using iterated integrals in a manner very similar to the case of $[0,1]$. For a simplicial decomposition of a manifold, the maps are locally defined on each simplex and can be glued together to extend the integral map to an $A_\infty$ morphism.

\section{Structure of a general $A_\infty$ morphism}

The Boolean cumulants can be defined for a map between $A_\infty$ algebras in multiple ways up to homotopy. Suppose $A$ and $B$ are $A_\infty$ algebras and $p$ is a chain map between them. In an $A_\infty$ algebra products of three or more variables are not well defined thus $(abc)$ can be defined as $(ab)c$ or $a(bc)$. Thus while there is only one way to define $K_1$ and $K_2$, there are four different ways of defining $K_3$. All the four ways of defining $K_3$ are homotopic to each other since $p((ab)c)$ is homotopic to $p(a(bc))$ and $p(a)(p(b)p(c))$ is homotopic to $(p(a)p(b))p(c)$. 

\begin{lemma}
Suppose $A$ and $B$ are $A_\infty$ algebras. The different ways of defining the cumulants are homotopic to each other via the maps $m_2$. Multiple homotopies given in this manner are all homotopic to each other, the homotopies of such homotopies are homotopic to each other and so on 
\end{lemma}

\begin{proof}
The terms of the cumulants that are defined only up to homotopy correspond to the vertices of Stasheff associahedra. The different ways of defining the cumulants are homotopic to each other via the edges. The two cells correspond to the homotopies of such homotopies and so on. Since the associahedra are contractible, the above lemma follows. 
\end{proof}

Now suppose $(p_1,p_2,\ldots)$ is an $A_\infty$ morphism from $A$ to $B$. The compatibility equation still implies that $p_2$ gives a homotopy between $p_1(ab)$ and $p_1(a)p_1(b)$. However we now have \[p_1((ab)c) \neq p_1(a(bc))\] \[\{p_1(a) p_1(b)\} p_1(c) \neq p_1(a) \{p_1(b) p_1(c)\}\]There are a triple products $m^A_3$ and $m^B_3$ on $A$ and $B$ respectively, which makes terms homotopic to each other. When $A$ and $B$ were associative, there were four different ways of combining three inputs from $A$ using $p_1$ and the binary products to give one output from $B$. When $A$ and $B$ are $A_\infty$ algebras there are six different ways that are now homotopic to each other via maps involving $p_2$, $p_1$, $m_2$ and $m_3$. 

\begin{lemma}
The cycle created by various homotopies between the several ways of combining three inputs is homotopic to zero via the homotopy $p_3$. 
\end{lemma}
\begin{proof}

Thus if we made a graph $G_3$ with six vertices each corresponding to ways of combining $n$ inputs, and edges corresponding to appropriate homotopies, we get a hexagon. Recall that the equation the $p_3$ satisfies gives the value of $\partial(p_3)$ to be

\[d(p_3)-p_3(\tilde{d})\]
\[= p_2(m_2\otimes 1+1\otimes m_2)-m_2(p_1\otimes p_2+p_2\otimes p_1) \]\[+ p_1(m_3) - m_3(p_1\otimes p_1) \otimes p_1 \]

Note that the six terms of the boundary $p_3$ correspond to homotopies between adjacent vertices of hexagon $G_3$. 

\begin{figure}[H]
\centering \includegraphics[width=5in]{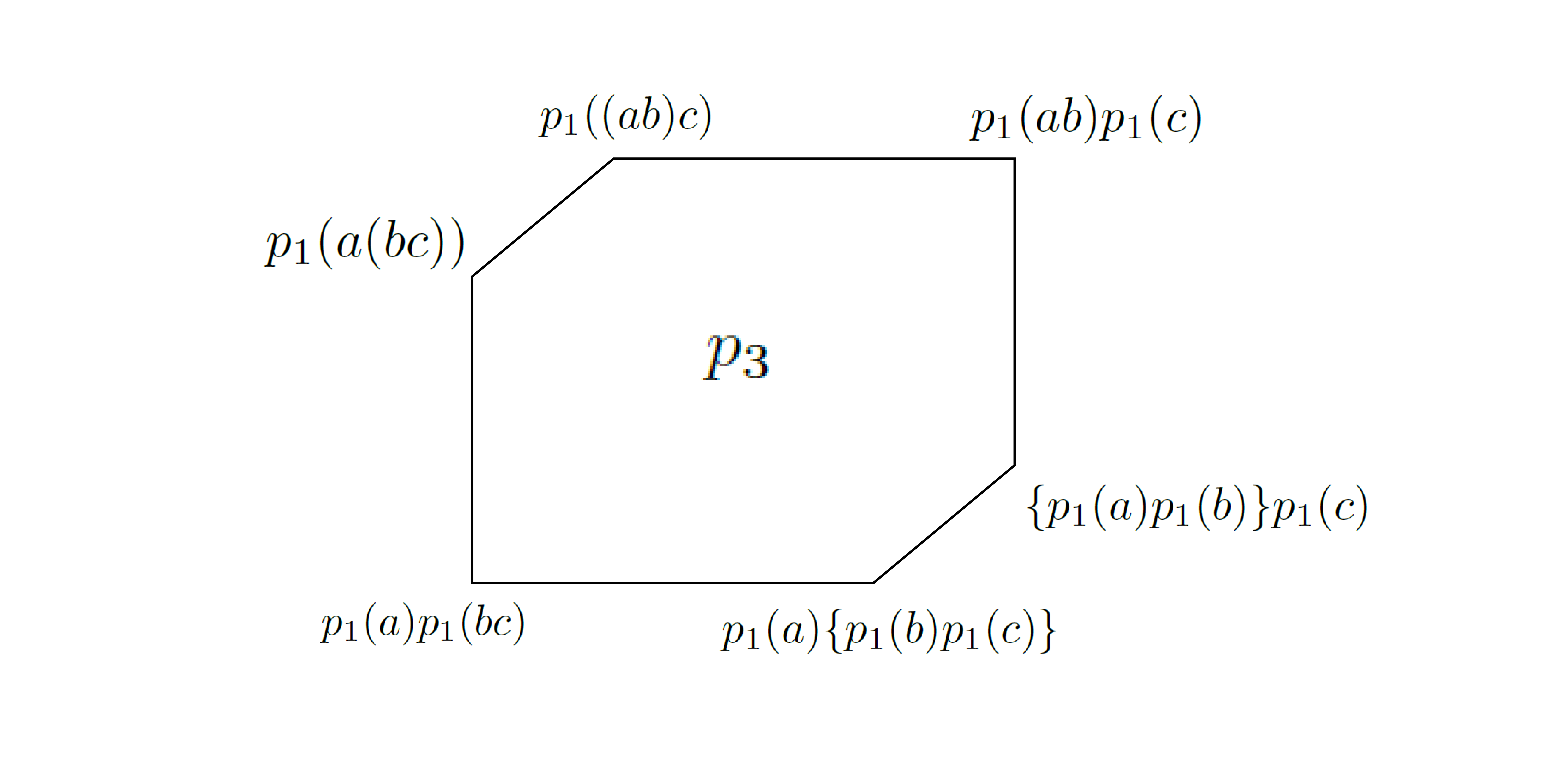}
\caption{}
\end{figure}

\end{proof}

Similarly for $K_4$ we get the following polyhedron 

\begin{figure}[H]
\centering \includegraphics[width=5in]{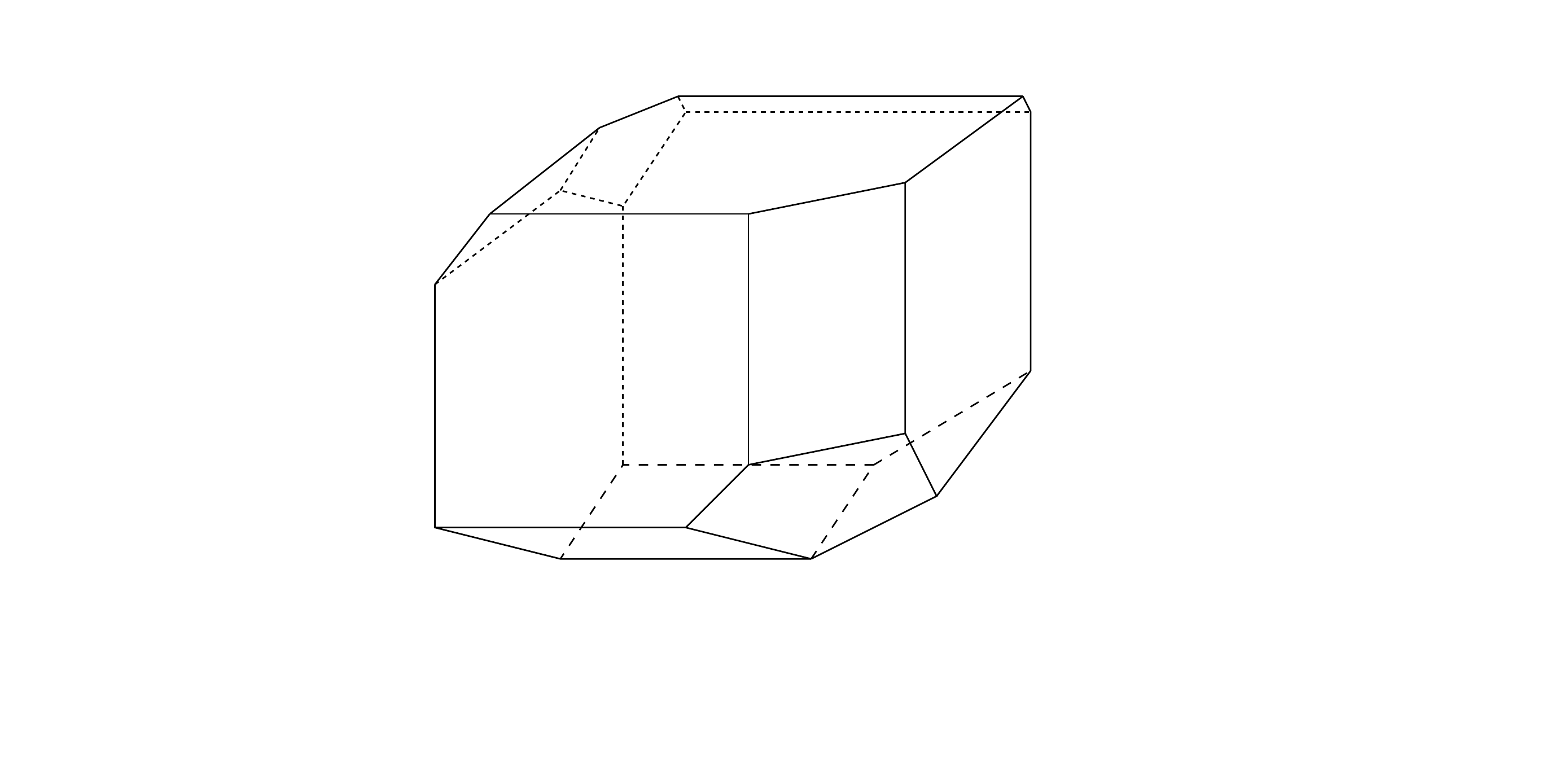}
\caption{}
\end{figure}

In the context of $A_\infty$ algebras the Boolean cumulants are only defined up to homotopy. In general for every $K_n$ there is an $n-1$ dimensional polyhedron whose cells correspond to maps which take $n$ inputs that are compositions of maps $p_j$'s and $m_j$'s. 

Since the Stasheff associahedra make these different ways homotopic to each other and indeed different homotopies are homotopic to each other and so on, we have the following theorem in the context of $A_\infty$ cumulants.

\begin{Theorem}
Let $A$ and $B$ be two $A_\infty$ algebras. Let $p$ be a chain map from $A$ to $B$. Let $K_2$, $K_3$ and so on be the Boolean cumulants of $p$ defined up to homotopy. Suppose $p$ is the first term of an $A_\infty$ morphism $(p, p_2, p_3, \ldots)$ where $p_n:A^{\otimes n}\rightarrow B$. Then the following statements hold.
\begin{itemize}
\item[i)] $p_2$ gives a homotopy from the second Boolean cumulant $K_2$ to zero. All the different ways of defining the higher Boolean cumulants $K_n$ are also homotopic to zero using maps created by $p_2$ and $p_1$.
\item[ii)] $p_3$ gives a homotopy between different ways of making $K_3$ homotopic to zero. For all the higher Boolean cumulants, homotopies between the multiple different ways of making them homotopic to zero are homotopic to each other using $p_3$, $p_2$ and $p_1$.
\item[iii)] In general any cycles that are created using the homotopies $\{p_j\}_{j=1}^n$ are made homotopic to zero using maps made by $\{p_j\}_{j=1}^{n+1}$.
\end{itemize}

\end{Theorem}

\begin{proof}
The proof of this theorem follows from the fact that the polyhedrons corresponding to each $p_n$ are contractible. 

\[\partial(p_n) = d(p_n) - p_n(d) \]\[= \sum_{k=2}^n \sum_{j=0}^{n-k} p_{n-k+1}(1\otimes \ldots m_k \ldots \otimes 1)\]\[\pm \sum_{k=2}^n \sum_{n_1+\ldots+n_k=n} m_k(p_{n_1}\otimes\ldots \otimes p_{n_k}) \]

Corresponding to each term in the above sum there is a face of the polyhedron.

The cells of the polyhedrons correspond to concrete maps constructed using $p_j$ and $m_j$ for smaller $j$. 

\end{proof}

\nocite{*}

\bibliography{Paper}
\bibliographystyle{plain}

\end{document}